\documentclass{amsproc}

\usepackage{tikz}

\newtheorem{theorem}{Theorem}[section]
\newtheorem{lemma}[theorem]{Lemma}

\theoremstyle{definition}

\newtheorem{example}[theorem]{Example}

\theoremstyle{remark}
\newtheorem{remark}[theorem]{Remark}

\numberwithin{equation}{section}

\begin{document}

\title{Hypergeometry and the AGM over Finite Fields}

\author{Eleanor McSpirit}
\address{Department of Mathematics, University of Virginia, Charlottesville, VA 22903}
\email{egm3zq@virginia.edu}
\thanks{E.M acknowledges the support of a UVa Dean's Doctoral Fellowship.}

\author{Ken Ono}
\address{Department of Mathematics, University of Virginia, Charlottesville, VA 22903}
\email{ko5wk@virginia.edu}
\thanks{K.O. thanks the Thomas Jefferson Fund and the NSF (DMS-2002265 and DMS-2055118) for their support.}

\subjclass[2020]{11G20, 14H52, 14K02, 33C90}

\date{}

\begin{abstract}
One of the most celebrated applications of Gauss' $_2F_1$ hypergeometric functions is in connection with the rapid convergence of sequences and special values that arise in the theory of arithmetic and geometric means. This theory was the inspiration for a recent paper \cite{jelly1} in which a finite field analogue of AGM$_\mathbb{R}$ was defined and then studied using finite field hypergeometric functions.  Instead of convergent sequences, one gets directed graphs that combine to form disjoint unions of graphs that individually resemble {\it jellyfish.} Echoing the connection of hypergeometric functions to periods of elliptic curves, these graphs organize elliptic curves over finite fields. Here we use such ``jellyfish swarms" to prove new identities for Gauss' class numbers of positive definite binary quadratic forms. Moreover, we prove that the sizes of jellyfish are in part dictated by the order of the prime above 2 in certain class groups. 
\end{abstract}

\maketitle


\section{Introduction and statement of results}

Recall that the classical \textit{arithmetic-geometric mean} iteration is defined for positive real numbers $a$ and $b$ by the sequence of pairs \[\text{AGM}_\mathbb{R}(a,b) :=\{(a_1,b_1), (a_2,b_2), \dots\},\] where $a_1:=a$, $b_1:=b$, and successive terms are given by
\[a_{n}:= \frac{a_{n-1}+b_{n-1}}{2} \quad \text{ and } \quad b_n:=\sqrt{a_{n-1}b_{n-1}}.
\] 
It is well known that both $(a_n)$ and $(b_n)$ rapidly converge to the same limit (p. 2, \cite{piagm}). One of the most famous results on the AGM$_\mathbb{R}$ is due to Gauss, who showed using the theory of elliptic integrals that one can generate extraordinary approximations for $\pi$ with relatively few iterations by considering the related sequence 
\[p_n := \frac{a_n^2}{1-\sum_{i=1}^n 2^{i-2}(a_i^2-b_i^2)}.
\]

In \cite{jelly1}, Griffin, Saikia, Tsai, and the second author defined a finite-field analogue of the $\text{AGM}_\mathbb{R}$ sequence over $\mathbb{F}_q$ when $q =p^r \equiv 3 \bmod{4}$. In this setting, $-1$ is not a square mod $q$, mirroring the fact that $-1$ is not a square in $\mathbb{R}$. This allows us to choose square roots such that the iterated geometric means are well-defined. That is, there is always a unique choice of $b_n = \sqrt{a_{n-1}b_{n-1}}$ such that $a_nb_n$ is a square mod $q$ when one starts with $a, b \in \mathbb{F}_q^\times$, $a \neq \pm b$, and $ab$ a square mod $q$.

For example, consider $q=7$ and $(a,b) = (4,2)$. Then
\[\text{AGM}_{\mathbb{F}_{7}}(4,2) = \{(4,2), \overline{(3,6), (1,2), (5,3), (4,1), (6,5), (2,4)}, \dots\}
\]
where the overlined pairs form a repeating orbit. The pairs $(6,3)$, $(2,1)$, $(3,5)$, $(1,4)$, and $(5,6)$ also enter this orbit after one AGM iteration. In \cite{jelly1}, the authors first explored the properties of the connected components of the directed graph representing the sequences of $\text{AGM}_{\mathbb{F}_q}(a,b)$ over all admissible pairs $(a,b)$. They showed that all components always consist of one cycle and one ``tentacle" of length one connected to each cycle vertex, for which they coin the name \textit{jellyfish}. For example, the following figure shows the unique connected component of the AGM$_{\mathbb{F}_{7}}$ graph $\mathcal{J}_{\mathbb{F}_{7}}$.
 
\begin{center}
 \begin{tikzpicture}
    \node (A) at (0,1) {\SMALL{(1,2)}};
    \node (A') at (0,2) {\SMALL{(6,3)}};
    \node (B) at (1,.5) {\SMALL{(5,3)}};
    \node (B') at (2,1) {\SMALL{(2,1)}};
    \node (C) at (1,-.5) {\SMALL{(4,1)}};
    \node (C') at (2,-1) {\SMALL{(3,5)}};
    \node (D) at (0,-1) {\SMALL{(6,5)}};
    \node (D') at (0,-2) {\SMALL{(1,4)}};
    \node (E) at (-1,-.5) {\SMALL{(2,4)}};
    \node (E') at (-2,-1) {\SMALL{(5,6)}};
    \node (F) at (-1,.5) {\SMALL{(3,6)}};
    \node (F') at (-2,1) {\SMALL{(4,2)}};

    \path [->] (A') edge node {} (A);
    \path [->] (B') edge node {} (B);
    \path [->] (C') edge node {} (C);
    \path [->] (D') edge node {} (D);
    \path [->] (E') edge node {} (E);
    \path [->] (F') edge node {} (F);
    \path [->] (A) edge node {} (B);
    \path [->] (B) edge node {} (C);
    \path [->] (C) edge node {} (D);
    \path [->] (D) edge node {} (E);
    \path [->] (E) edge node {} (F);
    \path [->] (F) edge node {} (A);
\end{tikzpicture} \hskip.35in  \includegraphics[width=2.25in]{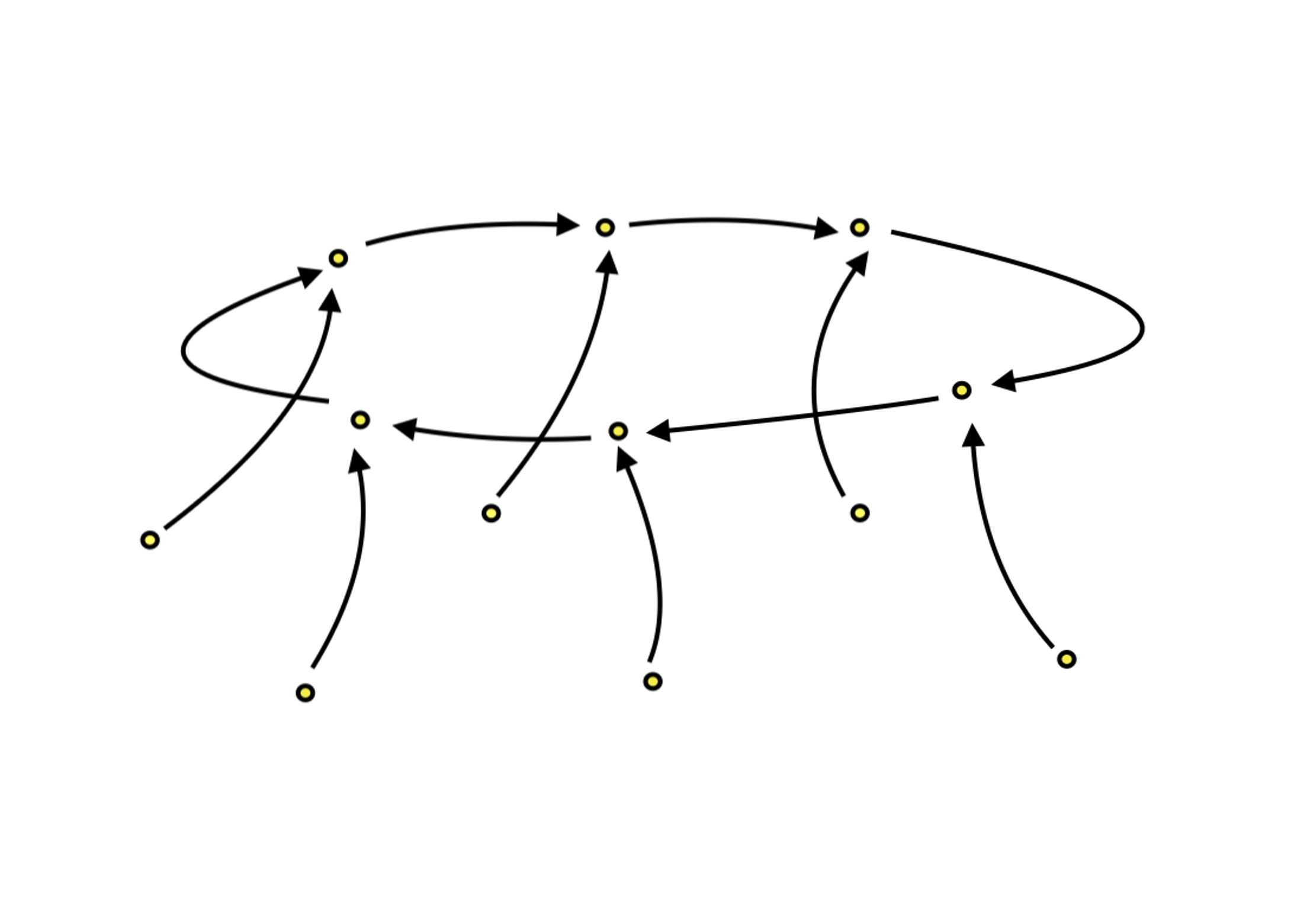}

\textsc{Figure 1}: Example of the jellyfish comprising $\mathcal{J}_{\mathbb{F}_{7}}$.
\end{center}
\medskip

In general, the graph $\mathcal{J}_{\mathbb{F}_{q}}$ consists of many such jellyfish, which together comprise a \textit{swarm}. In general, jellyfish swarms contain jellyfish of varying sizes and multiplicities, as exemplified by the following figure showing the swarm $\mathcal{J}_{\mathbb{F}_{19}}$. 

\begin{center}
\includegraphics[width=\textwidth]{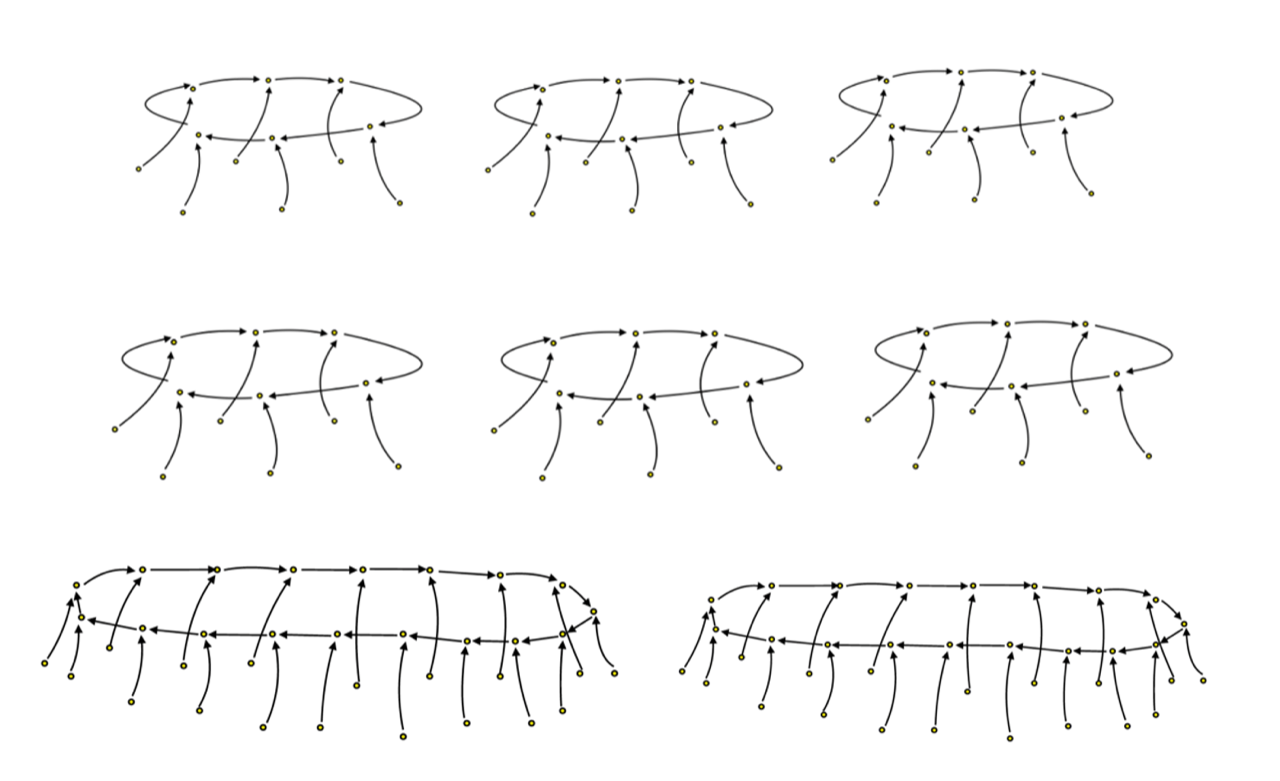}
\textsc{Figure 2}: Example of the jellyfish swarm $\mathcal{J}_{\mathbb{F}_{19}}$.
\end{center}

The theory underlying the classical $\text{AGM}_\mathbb{R}$ involves special integrals and their relationship with Gauss' hypergeometric functions. In particular, for $a>b>0$ we let 
\begin{equation} \label{eqn: ellipticintegral}
I_\mathbb{R}(a,b) := \frac{1}{2a} \int_1^\infty \frac{dx}{\sqrt{x(x-1)(x-(1-b^2/a^2))}}.
\end{equation}
A straightforward check shows that $I_\mathbb{R}(a,b) = I_\mathbb{R}(\frac{a+b}{2},\sqrt{ab})$, and so the sequence $\text{AGM}_\mathbb{R}(a,b) =\{(a_1,b_1), (a_2,b_2), \dots\}$ satisfies
\[I_\mathbb{R}(a_1,b_1) = I_\mathbb{R}(a_2,b_2) = \cdots = I_\mathbb{R}(a_n,b_n)=\dots\]
Gauss discovered the following beautiful formula for $I_\mathbb{R}(a,b)$ in terms of his classical $_2F_1$ hypergeometric function:
\begin{equation}\label{GaussIntegral} I_\mathbb{R}(a,b) = \frac{\pi}{2a} \cdot {_2F_1}^\text{class}\left( \begin{matrix}
 \frac{1}{2} & \frac{1}{2} \\
 & 1 	
 \end{matrix} \ \middle| \ 1- \frac{b^2}{a^2}\right).
\end{equation}
It turns out that these elliptic integrals and hypergeometric functions encode data about elliptic curves; for $\lambda \neq 0 ,1$ we may define the Legendre elliptic curve
\[ E_{\lambda}\colon \quad y^2 = x(x-1)(x-\lambda).\]
It turns out that the real period of $E_\lambda$ is computed by the integral
\[\Omega(E_\lambda) = \int_1^\infty \frac{dx}{\sqrt{x(x-1)(x-\lambda)}}.
\]
Gauss' $_2F_1$ hypergeometric functions offer us a closed formula for the real period of $E_\lambda$ when $0 < \lambda < 1$: 
\begin{equation}\label{eqn: real period} \Omega(E_\lambda) = \pi \cdot {_2F_1}^\text{class}\left( \begin{matrix}
 \frac{1}{2} & \frac{1}{2} \\
 & 1 	
 \end{matrix} \ \middle| \ \lambda \right).
 \end{equation}

Just as the classical $\text{AGM}_\mathbb{R}$ interacts with hypergeometric functions and elliptic curves, so too does the finite field analogue described above. 
In analogy with the elliptic integrals in (\ref{eqn: ellipticintegral}), one can define
\begin{equation}I_{\mathbb{F}_q}(a,b):= \frac{1}{2a} \sum_{x \in \mathbb{F}_q} \phi_q(x)\phi_q(x-1)\phi_q(x-(1-b^2/a^2)),
\end{equation}
where $\phi_q(-)$ denotes the quadratic character over $\mathbb{F}_q$. 
Greene's 1984 PhD thesis \cite{greene} offers us the appropriate analogue of Gauss' hypergeometric functions with which to draw our connection; 
for multiplicative characters $\{\alpha_i\}$, $\{\beta_j\}$ over $\mathbb{F}_q^\times$, he defined
\begin{equation*}\label{eqn: finitedef}
{_nF_{n-1}}\left(\begin{matrix}
\alpha_1 & \alpha_2 & ... & \alpha_n \\
& \beta_1 & ... & \beta_{n-1}
\end{matrix}
\ \middle| \ x \right)_q := \frac{q}{q-1}\sum_{\chi} {\alpha_1\chi\choose \chi}{\alpha_2\chi\choose \beta_1\chi} \cdots {\alpha_n\chi\choose \beta_{n-1}\chi}\chi(x),
\end{equation*}
where ${\alpha \choose \beta}$ is the normalized Jacobi sum $J(\alpha, \beta)$, defined by 
\[ {\alpha \choose \beta} := \frac{\beta(-1)}{q}J(\alpha, \bar{\beta}) := \frac{\beta(-1)}{q} \sum_{x \in \mathbb{F}_q} \alpha(x)\bar{\beta}(1-x).
\]

By comparing their definitions, one finds that for $q=p^r \equiv 3 \bmod{4}$

\begin{equation} I_{\mathbb{F}_q}(a,b) = \frac{q}{2a} \cdot {_2F_1}\left( \begin{matrix}
 \phi_q & \phi_q \\
 & \varepsilon_q	
 \end{matrix} \ \middle| \ 1- \frac{b^2}{a^2} \right)_q.\end{equation}
 This expression is the finite field analogue of (\ref{GaussIntegral}).
 Moreover, a recent result of Evans and Greene (Theorem 2 of \cite{evansgreene}) offers a quadratic transformation law for the finite field ${_2F_1}$ which implies that the sequence 
 $\text{AGM}_{\mathbb{F}_q}(a,b) =\{(a_1,b_1), (a_2,b_2), \dots\}$ satisfies
\[I_{\mathbb{F}_q}(a_1,b_1) = I_{\mathbb{F}_q}(a_2,b_2) = \cdots = I_{\mathbb{F}_q}(a_n,b_n)=\dots\]

 We may again consider the relationship of these hypergeometric functions to Legendre elliptic curves, this time analyzing the $\mathbb{F}_q$-points $E_\lambda(\mathbb{F}_q)$. In this case, instead of an elliptic integral telling us the period of $E_\lambda$, when $\text{char}(\mathbb{F}_q) \geq 5$ we have for $\lambda \in \mathbb{F}_q \backslash \{0,1\}$ that
 \[a_\lambda(q) = -\sum_{x \in \mathbb{F}_q} \phi_q((x)(x-1)(x-\lambda)),\]
 where $a_\lambda(q)$ is the \textit{trace of Frobenius} of $E_\lambda$ (see Theorem 11.10 of \cite{web}). Number theoretically, the trace of Frobenius is essentially a formula for the number of points on an elliptic curve over $\mathbb{F}_q$ since
 \[a_\lambda(q) = q+1 - \#E_\lambda(\mathbb{F}_q).
 \]
 Cast in terms of Greene's hypergeometric functions, we have in analogy with (\ref{eqn: real period})
\begin{equation} a_\lambda(q) = -q \cdot  {_2F_1}\left(\begin{matrix}
\phi_q, & \phi_q\\
& \varepsilon_q
\end{matrix}
\ \Big| \ \lambda \right)_q.
\end{equation}

The connection between the AGM$_{\mathbb{F}_q}$ and elliptic curves over $\mathbb{F}_q$ does not stop here, however. 
The preceding discussion has linked points on the graph of AGM$_{\mathbb{F}_q}$ to Legendre elliptic curves over $\mathbb{F}_q$ and has shown that two curves represented in the same connected component have the same trace of Frobenius, and hence the same number of $\mathbb{F}_q$-points. 
The authors of \cite{jelly1} are able to prove something stronger. They show that given a sequence $\text{AGM}_{\mathbb{F}_q}(a,b) =\{(a_1,b_1), (a_2,b_2), \dots\}$, if $\lambda_n := b_n^2/a_n^2$, then 
\[E_{\lambda_1}(\mathbb{F}_q) \cong E_{\lambda_2}(\mathbb{F}_q) \cong \cdots \cong E_{\lambda_n}(\mathbb{F}_q)\cong \cdots
\]
as abelian groups. 

Not only can the vertices of the AGM$_{\mathbb{F}_q}$ graphs be associated to elliptic curves, but it turns out that each iteration of the arithmetic-geometric mean, and hence each edge of the graph, corresponds to an isogeny of degree 2 between these curves (see Theorem 3 (2) of \cite{jelly1}). 
Following this observation, we can employ
tools from the study of elliptic curves in order to understand the structure of $\text{AGM}_{\mathbb{F}_q}$. Our approach will rely on the theory of complex multiplication. In particular, we will study the action of the class group of an imaginary quadratic order $\mathcal{O}$ on the set of isomorphism classes of elliptic curves over $\mathbb{F}_q$ with complex multiplication by $\mathcal{O}$. 

Our setup in some ways mirrors the robust theory of ``isogeny volcanoes," graphs which organize isomorphism classes of elliptic curves over finite fields and their isogenies. The seminal work on such graphs is due to Kohel\footnote{Kohel did not use this language to describe the graphs he studied. The terminology came later, in a paper by Fouquet and Morain, which applied the work of Kohel to the Schoof-Atkin-Elkies point-counting algorithm \cite{fandm}.}, who studied endomorphism rings of elliptic curves over finite fields by understanding the structure of 
$\ell$-isogeny graphs $\mathcal{G}_\ell(\mathbb{F}_q)$, which are visually quite similar to our jellyfish \cite{kohelthesis}.

However, there are two main differences between our isogeny graphs and the  graphs $\mathcal{G}_2(\mathbb{F}_q)$ studied by Kohel and others. First, only particular 2-isogenies between the represented elliptic curves will appear in our setting. Second, different vertices can correspond to elliptic curves in the same $\mathbb{F}_q$-isomorphism class. What's more, some jellyfish may in fact be identical after identifying nodes with their corresponding elliptic curves. These multiplicities will require attention for all of our results. 

In this note, we first show that $\text{AGM}_{\mathbb{F}_q}$ provides new information in the context of Gauss' theory of class numbers 
 of imaginary quadratic fields and class numbers of positive definite binary quadratic forms. To make this connection, we count the $\mathbb{F}_q$-isomorphism classes of elliptic curves that are represented in these graphs. To this end, the authors of \cite{jelly1} made use of a correspondence between isomorphism classes of elliptic curves with prescribed torsion and certain class numbers to count the number of distinct $j$-invariants which appear. 
 
 Recall that the Hurwitz class number $H(D)$, introduced by Adolf Hurwitz, is a modification of the class number of binary quadratic forms of discriminant $D \leq 0$. If $f$ is a quadratic form, then a matrix 
 \[\begin{bmatrix}
	\alpha & \beta \\
	\gamma & \delta 
\end{bmatrix} \in \text{SL}_2(\mathbb{Z}) \]
is an automorphism of $f$ if $f(\alpha x + \beta y , \gamma x + \delta y) = f(x,y)$. Then $H(D)$ weights forms of discriminant $D$ by $2/g$, where $g$ is the order of their automorphism group. We additionally declare $H(0) = -1/12$. 

 Following the approach of \cite{jelly1}, we obtain new class number formulas that are relatives of classical results like the Hurwitz-Kronecker class number formula
\[\sum_{|t| \leq 2\sqrt{N}}H(4N-t^2)=\sum_{d \mid N} \max\{d, N/d\},\]
which expresses sums of Hurwitz class numbers in terms of divisor sums. The above specializes to \[\sum_{|t| \leq 2\sqrt{p}}H(4p-t^2)=2p\] when $p$ is prime. 
Generalizations of the above were proven by Eichler and Zagier, the latter achieved through the construction of a weight 3/2 non-holomorphic Eisenstein series whose coefficients are Hurwitz class numbers \cite{Eichler}, \cite{Zagier}. 
More recently, Mertens \cite{mertens} analyzed the holomorphic projection of the Rankin-Cohen bracket of the Harmonic Maass form $\mathcal{H}(\tau)$ with certain theta functions in order to obtain weighted class number formulas and their asymptotics as $q \to \infty$. In \cite{OSS} Saad, Saikia, and the second author extended Mertens’ approach in their work on the Sato-Tate distribution for a certain family of $K3$ surfaces.

Additionally, many more such identities were recently proven using the trace formula and the combinatorics of $j$-invariants of elliptic curves over finite fields \cite{HCN}. 

Our addition to this area is the following: 

\begin{theorem}\label{thm: classnumber} 
Let $q=p^r \equiv 3 \bmod{4}$ where $p>3$. The following sums are taken over $t$ such that $|t| \leq 2 \sqrt{q}$.
\begin{enumerate}
\item If $q \equiv 3 \bmod{8}$, then we have
\[ q =3 + 4 \sum_{\substack{(t,q) = 1 \\ t \equiv q+1 (8) }} H\left( \frac{4q-t^2}{4}\right).
\]	
\item If $q \equiv 7 \bmod{8}$, then we have
\[ q = 3+ 4 \cdot \Bigg[ h(-q) + \sum_{\substack{(t,q)=1 \\ t \equiv q+1 (8)}} H\left( \frac{4q-t^2}{4}\right)\Bigg],
\]	
where $h(D)$ is denotes the class number of discriminant $D$. x

\end{enumerate}
\end{theorem}

\begin{remark}
The expressions in the two cases of Theorem \ref{thm: classnumber} have the same form for prime $q$. 
\end{remark}

There are many questions one can immediately ask about AGM$_{\mathbb{F}_q}$ once the ``jellyfish swarm" structure is known. For example, how many jellyfish $\mathcal{J}_i$ appear in a swarm? How large do we expect jellyfish to be?

Interestingly, these questions are not easily answered. To start, the number of jellyfish varies greatly across prime powers $q$. For example, we give the following table showing some values of $d(q)$, the number of jellyfish in $\mathcal{J}_{\mathbb{F}_q}$:
\vspace{.08in}
\begin{center}
\begin{tabular}{|c || c | c | c | c | c | c | c | c | c | c | c | c | c |} 
\hline
	$q$ & 7 & 11 & 19 & 23 & 27 & 31 & 43 & 47 & $\cdots$  & 161047 & 161051 & 161059 &  161071\\
	\hline
	$d(q)$ & 1 & 3 & 8 & 5 & 39 & 10 & 7 & 4 & $\cdots$  & 6499 & 25558635 & 4902 & 33744\\
	\hline
\end{tabular}
\vspace{.05in}

\textsc{Figure 3}: Selected values of $d(q)$ for prime powers $q$.
\vspace{.1in}
\end{center}

These questions are complicated by the fact that the sizes of jellyfish within a single swarm can vary widely. For example, for $q=161051=11^5$, the swarm $\mathcal{J}_{\mathbb{F}_{q}}$, one has tiny jellyfish of size 10 alongside those of massive size 7500. We offer the following table showing the size disparities between the smallest and largest jellyfish, which we denote by $\min(q)$ and $\max(q)$ respectively, for small values of $q$:
\vspace{.08in}
\begin{center}
\begin{tabular}{|c || c | c | c | c | c | c | c | c | c | c | c | c | c |}  
\hline
	$q$ & 7 & 11 & 19 & 23 & 27 & 31 & 43 & 47 & 59  & 67 & 71 & 79 & 83\\
	\hline
	$\min(q)$ & 12 & 10 & 12 & 22 & 6 & 30 & 42 & 230 & 174  & 18 & 28 & 52 & 410 \\
	\hline
	$\max(q)$ & 12 & 20 & 36 & 132 & 12 & 60 & 168 & 276 & 348  & 396 & 280 & 390 & 820 \\
	\hline
\end{tabular}
\vspace{.05in}

\textsc{Figure 4}: Selected values of $\min(q)$, $\max(q)$ for prime powers $q$.
\vspace{.1in}
\end{center}

Despite this complicated behavior, the authors of \cite{jelly1} use the fact that the trace of Frobenius is constant on a jellyfish to give a lower bound for $d(q)$ as a function of $q$, namely if $\varepsilon > 0$ then for sufficiently large $q$ we have
$$d(q) \geq \left(\frac{1}{2} -\varepsilon\right) \sqrt{q}$$ by counting the number of traces that must be obtained by curves in $\mathcal{J}_{\mathbb{F}_q}$.

In addition to size considerations for $d(q)$, one can also ask about the sizes of individual jellyfish. Here we show that these sizes are related to the algebraic properties of the endomorphism rings $\text{End}(E)$  of elliptic curves and their class groups (see Section 2 for background and definitions). Using Theorem \ref{thm: classnumber} and the theory of complex multiplication, we are able to show the following: 

\begin{theorem}\label{thm: jellysize}
Let $q =p^r \equiv 3 \bmod{4}$ where $p>3$. Suppose $(a,b)$ satisfies the conditions to appear in the AGM$_{\mathbb{F}_q}$ graph on the jellyfish $\mathcal{J}$, and let $\lambda := b^2/a^2$. Let $\# \mathcal{J}$ denote the number of vertices in $\mathcal{J}$. If $\mathcal{O} := \emph{End}(E_\lambda)$ and $h_2(\mathcal{O})$ denotes the order of $ [\mathfrak{p}_2]$ in $cl(\mathcal{O})$, where $\mathfrak{p}_2$ is a prime above $(2)$ in $\mathcal{O}$, then we have
\[ 2 \cdot h_2(\mathcal{O}) \Bigm| \#\mathcal{J}.
\]	
Additionally, if $m(\mathcal{J})$ denotes the multiplicity with which a jellyfish appears, then \[m(\mathcal{J}) \cdot \# \mathcal{J} = 2(q-1)\cdot h_2(\mathcal{O}).\]
\end{theorem}

\begin{example}
We illustrate the theorem with $q=271$. 
\vspace{.08in}
\begin{center}
\begin{tabular}{| c | c | c | c | c | c | }
	\hline
	$\mathcal{J}$ & $t$  & $h(\mathcal{O})$ & $h_2(\mathcal{O})$ & $\# \mathcal{J}$ & $m(\mathcal{J})$\\
	\hline
	$\mathcal{J}_1$ & $-32$  & 2 & 2 & $540$ & 2\\
	$\mathcal{J}_2$ & $-24$  &  5 & 5 & $900$ & 3\\
	$\mathcal{J}_3$ & $-16$  &  6 & 6 & $1620$ & 2\\
	$\mathcal{J}_4$ & $-16$  & 3 & 3 & $810$ & 2 \\
	$\mathcal{J}_5$ & $-8$  & 12 & 6 & $1620$ & 2\\
	$\mathcal{J}_6$ & $-8$  & 12 & 6 & $1620$ & 2\\
	$\mathcal{J}_7$ & $0$  &  11 & 11 & $2970$ & 2 \\
	$\mathcal{J}_8$ & $8$  &  12 & 6 & $1620$ & 2\\
	$\mathcal{J}_9$ & $8$ &  12 & 6 & $1620$ & 2 \\
	$\mathcal{J}_{10}$ & $16$  & 6 & 6 & $1620$ & 2\\
	$\mathcal{J}_{11}$ & $16$  &  3 & 3 & $810$ & 2\\
	$\mathcal{J}_{12}$ & $24$  &  5 & 5 & $2700$ & 1\\
	$\mathcal{J}_{13}$ & $32$  &  2 & 2 & $108$ & 10\\
	\hline
\end{tabular}
\vspace{.1in}
\end{center}
\end{example}

\begin{remark}
While $\displaystyle{2 \cdot h_2(\mathcal{O}) \mid\#\mathcal{J}}$, these two numbers are in general unequal, as illustrated in the previous example. The quantity $\displaystyle{\#\mathcal{J}/ 2 \cdot h_2(\mathcal{O})}$ is a divisor of $q-1$ whose appearance we will explain in Section \ref{subsection: taxonomy}. For now, it suffices to remark that it is not determined by the traces or endomorphism rings of the elliptic curves comprising $\mathcal{J}$. 
\end{remark}

This paper is organized as follows. Section \ref{subsection: endo} will recall necessary background from the theory of complex multiplication that will allow us to study vertices and edges in $\mathcal{J}_{\mathbb{F}_q}$ using class groups. In Section \ref{subsection: taxonomy}, we give background from \cite{jelly1} on the taxonomy of jellyfish. In particular, we characterize the number of vertices in a swarm, as well as the multiplicity of each isomorphism class of elliptic curves appearing on a jellyfish. In Section 4, we prove Theorems \ref{thm: classnumber} and \ref{thm: jellysize}.

\section{Complex Multiplication}\label{subsection: endo}

We first recall the classical theory of complex multiplication of elliptic curves  over $\mathbb{C}$ (for example, see Chapter 2 of \cite{silvermanadvanced}):
Let $\mathcal{O}$ be an order in an imaginary quadratic field. If $\mathfrak{a}$ is an invertible $\mathcal{O}$-ideal, then the torus $\mathbb{C}/\mathfrak{a}$ corresponds to an elliptic curve $E(\mathbb{C})$ with complex multiplication by $\mathcal{O}$. Equivalent ideals correspond to isomorphic elliptic curves, and we have a bijection between the ideal class group $cl(\mathcal{O})$ and the set 
\[\text{Ell}_{\mathcal{O}}(\mathbb{C}) := \{ j(E/\mathbb{C})\mid \text{End}(E) \cong \mathcal{O}\}
\]
of $j$-invariants of elliptic curves over $\mathbb{C}$ with complex multiplication by $\mathcal{O}$. 

Further, another invertible $\mathcal{O}$-ideal $\mathfrak{b}$ uniquely determines a separable isogeny of degree $N(\mathfrak{b})$ with kernel 
\[E[\mathfrak{b}] := \{P \in E \mid \alpha \cdot P = O \text{ for all } \alpha \in \mathfrak{b}\}\]
such that the target curve also has multiplication by $\mathcal{O}$. One can check that principal ideals act trivially, and that this defines a faithful $cl(\mathcal{O})$-action on $\text{Ell}_\mathcal{O}(\mathbb{C})$. 

While this correspondence is pleasing, we are concerned with the case where $\mathbb{C}$ is replaced by the finite field $\mathbb{F}_q$. It turns out that the story in this setting is largely the same. 

Let $E$ be an ordinary elliptic curve over $\mathbb{F}_q$, and let $\pi_E$ denote the Frobenius endomorphism of $E$. One may compute the trace of Frobenius to be $t = q+1 - \# E(\mathbb{F}_q)$. Using the characteristic equation for $\pi_E$, one derives the \textit{norm equation} 
\[t^2 -4q = v^2D_K,\] where $v^2D_K$ is the discriminant of the imaginary quadratic order $\mathbb{Z}[\pi_E]$ and $D_K$ is the discriminant of its field of fractions $K$. Then if $\mathcal{O} := \text{End}(E/\mathbb{F}_q)$, we have 
\[\mathbb{Z}[\pi_E] \subseteq \mathcal{O} \subseteq \mathcal{O}_K,\]
and $\mathcal{O}$ has discriminant $u^2 D_K$, where $u = [\mathcal{O}_K : \mathcal{O}]$ divides $v =[\mathcal{O}_K : \mathbb{Z}[\pi_E]]$. 

Now, consider the set \[\text{Ell}_t(\mathbb{F}_q) := \{j(E/\mathbb{F}_q) \mid \text{tr}(\pi_E) = t\}\]
of $\overline{\mathbb{F}}_q$-isomorphism classes of elliptic curves over $\mathbb{F}_q$ with trace of Frobenius $t$. Tate's Isogeny Theorem \cite{tate} implies that $\text{Ell}_t(\mathbb{F}_q)$ determines an isogeny class. Further, as $K$ is determined by $t$ and $q$, this set can be written as the disjoint union 
\[\text{Ell}_t(\mathbb{F}_q) = \bigsqcup_{\mathbb{Z}[\pi_E] \subseteq \mathcal{O} \subseteq \mathcal{O}_K} \text{Ell}_{\mathcal{O}}(\mathbb{F}_q),
\]
where $\text{Ell}_{\mathcal{O}}(\mathbb{F}_q)$ is defined in the same way as $\text{Ell}_{\mathcal{O}}(\mathbb{C})$. 
As a consequence of the Deuring Lifting Theorem, the norm equation implies that over $\mathbb{F}_q[x]$, the Hilbert class polynomial $H_{u^2D_K}(x)$ of degree $h(\mathcal{O})$ splits completely and the roots are precisely the set $\text{Ell}_\mathcal{O}(\mathbb{F}_q)$.
Then so long as $\text{Ell}_\mathcal{O}(\mathbb{F}_q)$ is nonempty, the set has cardinality $h(\mathcal{O})$.

Recall the definition of the \textit{Hurwitz class number} for an imaginary quadratic order $\mathcal{O}$:
\begin{equation} H(\mathcal{O}) :=\sum_{\mathcal{O} \subseteq \mathcal{O}' \subseteq \mathcal{O}_K} h(\mathcal{O}')
\end{equation}
If $D$ is the discriminant of $\mathcal{O}$, we may define $H(D) := H(\mathcal{O})$. This agrees with the definition given earlier. From the results above we immediately have that the cardinality of $\text{Ell}_t(\mathbb{F}_q)$ is equal to $H(t^2-4q)$. 

As in the characteristic 0 case, we again have a faithful action of $cl(\mathcal{O})$ on $\text{Ell}_\mathcal{O}(\mathbb{F}_q)$. What's more, if $\varphi \colon E \to E'$ is an isogeny of degree $\ell$ such that $\mathcal{O} := \text{End}(E) = \text{End}(E')$, then $\varphi$ results from the action of an invertible $\mathcal{O}$-ideal $\mathfrak{l}$ of norm $\ell$.
This action is what will ultimately allow us to compare edges in $\mathcal{J}_{\mathbb{F}_q}$ with elements of class groups.

\section{Taxonomy of Jellyfish}\label{subsection: taxonomy}

Here we give several necessary results on the structure of jellyfish swarms. We first recall results from \cite{jelly1} about the vertices in a jellyfish swarm. 
\begin{theorem}[Theorems 1 (2) and 3 (1) of \cite{jelly1}] 
\label{thm: counts} The following are true.
\begin{enumerate} 
\item The jellyfish swarm $\mathcal{J}_{\mathbb{F}_q}$ has $(q-3)(q-1)/2$ vertices. 
\item Each $E_\lambda$ for $\lambda \in \mathbb{F}_q^{\times 2} \backslash \{0,1\}$ appears exactly $q-1$ times.  
\end{enumerate}
\end{theorem}

The proof of (1) counts the number of admissible pairs $(a,b)$, and the proof of (2) relies on the fact that if $(a,b)$ corresponds to an elliptic curve $E_\lambda$, then so does every pair $(ka,kb)$ for $k \in \mathbb{F}_q^\times$. 

The above imply that the number of distinct $\lambda$ that occur in $\mathcal{J}_{\mathbb{F}_q}$ is $(q-3)/2$. However, we have alternative ways of counting the $\lambda$ that appear using the tools developed in Section 3 of \cite{jelly1}. 
In particular, recall that the $j$-invariants parameterize $\overline{\mathbb{F}}_q$-isomorphism classes of elliptic curves over $\mathbb{F}_q$, and that the trace and $j$-invariant uniquely determine an isomorphism class over $\mathbb{F}_q$. Further, if two elliptic curves have the same $j$-invariant, then they are either isomorphic over $\mathbb{F}_q$ or they are nontrivial quadratic twists of one another, in which case their traces differ by a sign. 
 If we are able to count the possible traces and the possible $j$-invariants corresponding to each trace, as well as the number of distinct $\lambda$ corresponding to each $j$-invariant, we will be able to get an alternative formula for the number of $\lambda$ which occur. 
 
 We first characterize the admissible traces of the elliptic curves $E_\lambda$. By Lemma 2 of \cite{jelly1}, one has that as an abelian group, $\mathbb{Z}/2\mathbb{Z} \times \mathbb{Z}/4\mathbb{Z} \subseteq E(\mathbb{F}_q)$. Then $8 \mid \# E(\mathbb{F}_q)$, and in particular $t \equiv q+1 \bmod 8$. We further know that every trace is represented so long as the Hasse bound $|t| \leq 2\sqrt{q}$ is satisfied (see \cite{Deuring} or \cite{ruck}). 

Now we fix a trace $t$. Since $\text{Ell}_t(\mathbb{F}_q)$ defines an isogeny class, every elliptic curve on a jellyfish must have the same trace. Then consider the set of jellyfish with trace $t$
 and define $M_{\mathbb{F}_q}(t)$ to be the number of distinct $j$-invariants across the union of these jellyfish. The authors of \cite{jelly1} prove the following:

\begin{theorem}[Theorem 6 of \cite{jelly1}]\label{thm: hurwitz3mod8}
Suppose $q \equiv 3 \bmod{8}$ and $p > 3$. If $|t| \leq 2\sqrt{q}$ such that $(t,p)=1$ and $t \equiv q+1 \bmod{8}$, then we have 
\[H\left(\frac{t^2-4q}{4}\right) = M_{\mathbb{F}_q}(t).
\]	
\end{theorem}

\begin{remark}
The careful reader may note that our characterization of Theorem 6 of \cite{jelly1} includes the extra assumption that $(t,p)=1$. This is not strictly necessary in the case where $q \equiv 3 \bmod{8}$, as all appearing traces will have this property, but will be important for the case where $q \equiv 7 \bmod{8}$ and so we include it here to avoid possible confusion. 
\end{remark}

The above relies primarily on the correspondence between the number of distinct $j$-invariants of elliptic curves with trace $t$ and the Hurwitz class number $H(t^2-4q)$ discussed in Section \ref{subsection: endo} with the extra observation that the removal of a factor of $2$ from the conductor $u$ is equivalent to the requirement that $E[2] \subset E(\mathbb{F}_q)$ \cite{schoof}, which must be satisfied as all of our elliptic curves admit a Legendre normal form. Since $q \equiv 3 \bmod{4}$, one can also show that every elliptic curve over $\mathbb{F}_q$ is of the form $E_\lambda$ for $\lambda \in \mathbb{F}_q^{\times 2} \backslash \{0,1\}$ by 2-descent.

By the same arguments as that which prove Theorem \ref{thm: hurwitz3mod8}, we get the following for the ordinary traces when $q \equiv 7 \bmod{8}$:

\begin{theorem}
Suppose $q \equiv 7 \bmod{8}$ and $p > 3$. If $|t| \leq 2\sqrt{q}$ such that $(t,p)=1$ and $t \equiv q+1 \bmod{8}$, then we have
\[H\left(\frac{t^2-4q}{4}\right) = M_{\mathbb{F}_q}(t)
\]
\end{theorem}

We now turn to the case where $E$ is supersingular (i.e. where $t=0$, as we shall see). While we need to amend the correspondence to avoid orders whose conductors are not coprime to $p$, we end up with a similar result: 

\begin{lemma} \label{lem: supersingular}
Suppose $q \equiv 7 \bmod{8}$ and $p > 3$. Then we have $h(-q) = M_{\mathbb{F}_q}(0)$.
\end{lemma}

\begin{proof} 
When $t=0$, the endomorphism ring of $E$ can be identified with an imaginary quadratic order in $\mathbb{Q}(\pi_E)$ with conductor prime to $p$ (Theorem 4.1, \cite{waterhouse}) containing $\mathcal{O}(-q)$ (Proposition 3.7, \cite{schoof}). 
Since $q$ is a power of $p$, there is only one such order, $\mathcal{O}(-q)$ itself. 
\end{proof}

Now, to justify our focus on these particular traces, we offer the following classification of the traces of elliptic curves over $\mathbb{F}_q$:

\begin{theorem}[Theorem 4.1 of \cite{waterhouse}]
	Let $q=p^k$ be a power of a prime $p$. Let $t \in \mathbb{Z}$ and let $N=q+1-t$. The integer $N$ is the cardinality of $E(\mathbb{F}_q)$ for some elliptic curve $E/\mathbb{F}_q$ if and only if one of the following conditions is satisfied:
	\begin{enumerate}
	\item $|t| \leq 2\sqrt{q}$ and $(t,p) = 1$;
	\item $k$ is odd and $t=0$;
	\item $k$ is odd, $t = \pm \sqrt{pq}$, and $p=2$ or 3;
	\item $k$ is even, $t = 0$, $p \not\equiv 1 \bmod{4}$;	
	\item $k$ is even, $t = \pm \sqrt{q}$, $p \not\equiv 1 \bmod{3}$;
	\item $k$ is even, $t = \pm 2\sqrt{q}$.
	\end{enumerate}
\end{theorem}

Since $q \equiv 3 \bmod{4}$ and $p > 3$, the only options are (1) and (2), meaning the above completely classify the cases where $M_{\mathbb{F}_q}(t)$ can be nonzero. 
Now that we have determined the number of $j$-invariants appearing in $\mathcal{J}_{\mathbb{F}_q}$ using class numbers, we return to the question of how many $\lambda$ correspond to a particular $j$. We offer the following answer:

\begin{lemma} \label{lem: lambdas}
	Let $L_{\mathbb{F}_q}(t,j)$ denote the number of distinct $\lambda$ such that $E_\lambda$ has trace $t$ and $j$-invariant $j$. Then for all pairs $(t, j)$ such that $L_{\mathbb{F}_q}(t,j) \neq 0$, we have $L_{\mathbb{F}_q}(t,j)= 2$.
\end{lemma}

\begin{proof}
First fix $t$ and $j$ and recall that there are exactly six $\lambda$ in $\overline{\mathbb{F}}_q \backslash \{0,1\}$ corresponding to each $j$-invariant not equal to 0 or 1728 (see \cite{silverman} Section III.1). What's more, given one such $\lambda$ we can find expressions for the other five by considering the orbit of $\lambda$ under the group generated by the transformations $\lambda \mapsto 1/\lambda$ and $\lambda \mapsto 1- \lambda$ on $\mathbb{P}^1$. We can write this set as 
\[ [\lambda] := \{\lambda, 1/\lambda, 1-\lambda, 1/(1-\lambda), \lambda/(\lambda -1), (\lambda-1)/\lambda\}.\]
Now suppose that $\lambda \in \mathbb{F}_q^{\times 2} \backslash \{1\}$ and that $E_\lambda$ has trace $t$ and $j$-invariant $j$. We first note that all other elements of $[\lambda]$ are also in $\mathbb{F}_q^\times\backslash \{1\}$. 
We must also check which are squares. If $\lambda$ is a square, then $1/\lambda$ is a square as well. Moreover, since $E_\lambda^{(\lambda)} \cong E_{1/\lambda}$ and $\lambda$ is a square, this twist is trivial and the trace and $j$-invariant are unchanged. Note that only one of $1-\lambda$ and $\lambda-1$ can be a square since $-1$ is not a square in $\mathbb{F}_q$. Then only one of $\{1-\lambda, 1/(1-\lambda)\}$ and $\{\lambda/(\lambda-1), (\lambda-1)/\lambda\}$ is a set of squares. 
One also has that $E_\lambda^{(-1)} \cong E_{1-\lambda}$ and $E_{1/\lambda}^{(-1)} \cong E_{(\lambda-1)/\lambda}$. Since $-1$ is not a square, both of these are nontrivial twists, and so do not have the same trace as $E_\lambda$. Then the only contributions to $L_{\mathbb{F}_q}(t,j)$ are $E_\lambda$ and $E_{1/\lambda}$. 

Now we deal with the two exceptional cases. First, suppose $E_\lambda$ has trace $t$ and satisfies $j(E_\lambda)=0$. Then $\lambda$ satisfies $\lambda^2-\lambda+1 = 0$ by the equation 
\[j(E_\lambda) = 2^8 \frac{(\lambda^2-\lambda+1)^3}{\lambda^2(\lambda-1)^2}.\]
Moreover, the other solution to $x^2-x+1=0$ is $1/\lambda$, so there are exactly two Legendre elliptic curves with trace $t$ and $j$-invariant 0. 

Finally, when $j(E_\lambda) = 1728$, we have that $\lambda \in \{-1, 2, 1/2\}$. We know that $-1$ is not a sqaure in our setting, so either both $2$ and $1/2$ are squares or neither are. Thus, $L_{\mathbb{F}_q}(t,j)=0$ or 2. 
\end{proof}

\section{Proofs of Theorems \ref{thm: classnumber} and \ref{thm: jellysize}}

\begin{proof}[Proof of Theorem 1.1]
We will demonstrate the argument for the case when $q \equiv 3 \bmod{8}$. The case when $q \equiv 7 \bmod{8}$ is the same except for the care needed to deal with the supersingular elliptic curves, for which one applies Lemma~\ref{lem: supersingular}. 
For $q \equiv 3 \bmod{8}$, the theorem is equivalent to proving that 
\[\frac{(q-1)(q-3)}{2} = \sum_{\substack{|t| \leq  2\sqrt{q} \\ t \equiv q+1 (8)}} 2(q-1) \cdot H\left( \frac{4q-t^2}{4}\right).
\]
The left hand side is the number of vertices in $\mathcal{J}_{\mathbb{F}_q}$, so it suffices to show the right hand side is also a count of the vertices in this graph. The sum runs over all the admissible traces of elliptic curves in $\mathcal{J}_{\mathbb{F}_q}$, so it suffices to show that each summand is the count of the number of vertices of trace $t$. Fixing $t$, each $j$ invariant defines a curve up to $\mathbb{F}_q$-isomorphism. The number of such $j$ is counted by $H\left(\frac{t^2-4q}{4}\right)$ by Theorem \ref{thm: hurwitz3mod8}. Corresponding to each isomorphism class, there are exactly 2 such $\lambda$ by Lemma~\ref{lem: lambdas}. Finally, by Theorem \ref{thm: counts} (2), each $\lambda$ appears in the graph with multiplicity $q-1$. 
\end{proof}

\begin{proof}[Proof of Theorem 1.3] 

By Theorem 3 (2) of \cite{jelly1}, every edge corresponds to the unique isogeny with kernel generated by $\langle (0, 0) \rangle$. In particular, this isogeny has degree 2. 

By Corollary 4 (1) of \cite{jelly1}, the groups $E(\mathbb{F}_q)$ are isomorphic for all $E$ on a jellyfish, and so their endomorphism rings are all the same; call it $\mathcal{O}$. This implies that this isogeny is the image of the class of an ideal $\mathfrak{p}_2$ of norm 2 in $\mathcal{O}$ with $E[\mathfrak{p}_2] = \langle(0,0)\rangle$. 
Since $E_{\lambda}$ and $E_{1/\lambda}$ are isomorphic over $\mathbb{F}_q$, the edge emanating from vertices corresponding to both of these curves must have the same target curve. In particular, every vertex at the end of a tentacle has a corresponding vertex in the cycle with the same $j$-invariant. Thus, it suffices to consider the action of $[\mathfrak{p}_2]$ on the cycles of the jellyfish with endomorphism ring $\mathcal{O}$. 
Since this action is faithful and all $j$-invariants in $\text{Ell}_\mathcal{O}(\mathbb{F}_q)$ are represented on some cycle, every cycle must have length divisible by the order of $[\mathfrak{p}_2]$ in $cl(\mathcal{O})$.  
Since the size of a jellyfish is twice the length of its cycle due to its tentacles of length one, our count must be multiplied by 2. 

To see that $m(\mathcal{J}) \cdot \# \mathcal{J} =2(q-1) \cdot h_2(\mathcal{O})$, note that the action of $[\mathfrak{p}_2]$ partitions the set of vertices into $[cl(\mathcal{O}):\langle[\mathfrak{p}_2]\rangle]$ subsets of size $2(q-1) \cdot h_2(\mathcal{O})$ by Theorem \ref{thm: classnumber}. 
Now if $n$ is the order of $[\mathfrak{p}_2]$ in $cl(\mathcal{O})$, then $[\mathfrak{p}_2^n]\cdot E_{\lambda} = E_{\lambda}$, and so if $\lambda = b^2/a^2$, the $n$th pair in the sequence of $\text{AGM}_{\mathbb{F}_q}(a,b)$ is $(ka,kb)$ for some $k$ in $\mathbb{F}_q^\times$. Then if $\text{ord}(k)$ denotes the multiplicative order of $k$ in $\mathbb{F}_q$, each jellyfish must have size $2\cdot h_2(\mathcal{O}) \cdot \text{ord}(k)$ and multiplicity $(q-1)/\text{ord}(k)$ by partitioning by the orbits of the action of $\langle k \rangle$. 

\end{proof}

\bibliographystyle{amsplain}
\bibliography{Hyper_Jelly_Draft.bib}

\end{document}